\documentclass[11pt,a4paper]{amsart}
\usepackage{amssymb,xspace}
\usepackage{amstext}
\usepackage[mathscr]{eucal}
\theoremstyle{plain}
\usepackage{amsbsy,amssymb,amsfonts,latexsym,eucal,amscd}
\usepackage[dvips]{graphicx}
\usepackage{epsfig}
\usepackage[all]{xy}
\usepackage{pstricks}
\newcommand{\ds }{\ensuremath{\displaystyle}}

\newcommand{\R }{\ensuremath{\mathbb R}}
\newcommand{\C }{\ensuremath{\mathbb C}}
\newcommand{\Q }{\ensuremath{\mathbb Q}}
\newcommand{\Z }{\ensuremath{\mathbb Z}}
\newcommand{\N }{\ensuremath{\mathbb N}}

\newcommand{\T }{\ensuremath{\mathbb T}}

\newcommand{\HHH }{\ensuremath{\mathbb H}}

\DeclareMathOperator{\rel}{\vphantom{y}rel}

\DeclareMathOperator{\codim}{codim}

\DeclareMathOperator{\ch}{ch}

\DeclareMathOperator{\supp}{supp}

\DeclareMathOperator{\pr}{pr}

\newcommand{\aaa }{\ensuremath{\mathcal{A}}}

\newcommand{\oo }{\ensuremath{\mathcal{O}}}
\newcommand{\hh }{\ensuremath{\mathcal{H}}}

\newcommand{\ii }{\ensuremath{\mathcal{I}}}

\newcommand{\LL }{\ensuremath{\mathcal{L}}}

\newtheorem{theorem}{Theorem}[section]

\newtheorem{lemma}[theorem]{Lemma}

\newtheorem{proposition}[theorem]{Proposition}

{\theoremstyle{definition}}

{\theoremstyle{definition}}

{\theoremstyle{definition}}

{\theoremstyle{definition}}

{\theoremstyle{definition}\newtheorem{definition}[theorem]{Definition}}

{\theoremstyle{definition}}

{\theoremstyle{definition}\newtheorem{remark}[theorem]{Remark}}

\address{Institut de Math\'{e}matiques de Jussieu, UMR 7586\\
Case 247\\ Universit\'{e} Pierre et Marie Curie\\
4, place Jussieu\\
F-75252 Paris Cedex 05\\
France}
\email{jgrivaux@math.jussieu.fr}

\newcommand{\xnj}[1]{X^{[n]}_{J^{\re}_{#1}}}

\newcommand{\oti}{\ensuremath{\otimes}}

\newcommand{\he}{^{\vphantom{[n]}} }
\newcommand{\be}{_{\vphantom{[n]}} }

\newcommand{\ba}[1]{\ensuremath{\overline{#1}}}

\newcommand{\ti }[1]{\ensuremath{\widetilde{#1}}}

\newcommand{\rb }{\ensuremath{\raisebox}}

\newcommand{\tim }{\ensuremath{\times}}

\newcommand{\ee }{\ensuremath{^{\, *}}}

\newcommand{\bop }{\ensuremath{\bigoplus\limits}}

\newcommand{\suq }{\ensuremath{\subseteq}}

\entrymodifiers={+!!<0pt,\fontdimen22\textfont2>}

\def\apl#1#2#3{#1\mkern -2 mu:\mkern - 4 mu
\xymatrix@C=17pt{#2\!\ar[r]&\!#3}
}

\def\aplpt#1#2#3#4{#1\mkern -4 mu:\mkern - 8 mu
\xymatrix@C=17pt{#2\!\ar[r]&\!#3#4}
}

\def\sutrgd#1#2#3{
\xymatrix@C=17pt{
0\ar[r]&#1\ar[r]&#2\ar[r]&#3\ar[r]&0
}
}

\def\sutrgdpt#1#2#3#4{
\xymatrix@C=17pt{
0\ar[r]&#1\ar[r]&#2\ar[r]&#3\ar[r]&0#4
}
}

\def\sutrgpt#1#2#3#4{
\xymatrix@C=17pt{
0\ar[r]&#1\ar[r]&#2\ar[r]&#3#4
}
}

\def\sutr#1#2#3{
\xymatrix@C=17pt{
#1\ar[r]&#2\ar[r]&#3
}
}

\def\sutrg#1#2#3{
\xymatrix@C=17pt{
0\ar[r]&#1\ar[r]&#2\ar[r]&#3
}
}

\def\sutrd#1#2#3{
\xymatrix@C=17pt{
#1\ar[r]&#2\ar[r]&#3\ar[r]&0
}
}

\def\fl{\xymatrix@C=17pt{
\ar[r]&
}}

\def\flcourte{\xymatrix@C=10pt{
\ar[r]&
}}

\def\flgd#1#2{\xymatrix@C=17pt{#1\!
\ar[r]&\!#2
}}

\def\flcourtegd#1#2{\xymatrix@C=15pt{\!\!#1\!
\ar[r]&\!#2\!\!
}}

\def\flgdin#1#2{\xymatrix@C=3ex{\!\!\scriptstyle{#1}
\ar[r]&\scriptstyle{#2}\!\!\!
}}

\def\fldouble{\xymatrix@1{
\ar@{->>}[r]&
}}

\def\fle#1#2{
\xymatrix@1{
#1
\ar[r]&#2
}}

\def\flex#1#2#3{
{\xymatrix@1{
#1
\ar[r]^{#3}&#2
}}
}

\def\fledouble#1#2{
{\xymatrix@1{
#1
\ar@{->>}[r]&{#2}
}}
}

\def\flexdouble#1#2#3{
{\xymatrix@1{
#1
\ar@{->>}[r]^{#3}&{#2}
}}
}

\def\diagca#1#2#3#4#5#6#7#8{\xymatrix@1{
#1
\ar[d]_{#6}\ar[r]_{#5}&#2\ar[d]_{#7}\\
#3
\ar[r]_{#8}&#4
}}

\def\sutrois#1#2#3{
{\xymatrix@1{
#1
\ar[r]&#2
\ar[r]&#3
}}
}

\def\sutroiszerogdprime#1#2#3{
{\xymatrix@1{
0
\ar@<-0.5mm>[r]&#1
\ar@<-0.5mm>[r]&#2
\ar@<-0.5mm>[r]&#3
\ar@<-0.5mm>[r]&0
}}
}

\def\fleprime#1#2{
\xymatrix@1{
#1
\ar[r]&#2
}}

\def\sutroisnom#1#2#3#4#5{
{\xymatrix@1{
#1
\ar[r]^{#4}&#2
\ar[r]^{#5}&#3
}}
}

\def\sutroiszerogd#1#2#3{
{\xymatrix@1{
0
\ar[r]&#1
\ar[r]&#2
\ar[r]&#3
\ar[r]&0
}}
}
\def\strgdexp#1#2#3#4#5#6{
{\xymatrix@1{
0
\ar[r]&\rb{#2ex}{$#1$}
\ar[r]&\rb{#4ex}{$#3$}
\ar[r]&\rb{#6ex}{$#5$}
\ar[r]&0
}}
}

\def\sutroiszerog#1#2#3{
{\xymatrix@1{
0
\ar[r]&#1
\ar[r]&#2
\ar[r]&#3
}}
}

\def\suxtroiszerogd#1#2#3#4#5{
{\xymatrix@1{
0
\ar[r]&#1
\ar[r]^{#4}&#2
\ar[r]^{#5}&#3
\ar[r]&0
}}
}

\def\suquatre#1#2#3#4{
{\xymatrix@1{
#1
\ar[r]&#2
\ar[r]&#3
\ar[r]&#4
}}
}

\def\suxquatre#1#2#3#4#5#6#7{
{\xymatrix@1{
#1
\ar[r]^{#5}&#2
\ar[r]^{#6}&#3
\ar[r]^{#7}&#4
}}
}

\def\sucinq#1#2#3#4#5{
{\xymatrix@1{
#1
\ar[r]&#2S^{n}X
\ar[r]&#3
\ar[r]&#4
\ar[r]&#5
}}S^{n}X
}

\def\suxcinq#1#2#3#4#5#6#7#8#9{
{\xymatrix@1{
#1
\ar[r]^{#6}&#2
\ar[r]^{#7}&#3
\ar[r]^{#8}&#4
\ar[r]^{#9}&#5
}}
}

\DeclareMathOperator{\id}{id}

\newcommand{\xg}{\ensuremath{\mathfrak{X} }}
\newcommand{\re}{\ensuremath{\, \rel}}
\newcommand{\ci}{\ensuremath{\mathcal{C}^{\mkern 1 mu\infty\vphantom{_{p}}}}}

\newcommand{\snx}{S^{\mkern 1 mu n\!}\be X}

\newcommand{\sn}{S^{\mkern 1 mu n\!}\be}
\newcommand{\xn}{X^{[n]}\be}
\newcommand{\hb}{H\ee\be}
\newcommand{\enn}{E^{\, [n] }\be}
\newcommand{\xb}{\ub{x}}
\newcommand{\yb}{\ub{y}}
\newcommand{\zb}{\ub{z}}
\newcommand{\ub}[1]{\underline{\vphantom{!}\vphantom{y}#1}}
\newcommand{\jre}{J^{\re }\be}
\newcommand{\got}{G\"{o}ttsche}
\newcommand{\nbh}{neighbourhood}
\newcommand{\npn}{[n'\!\!,\, n]}
\newcommand{\xna}[1]{X^{[#1]}\be}
\newcommand{\xnb}[1]{X^{#1}\be}
\newcommand{\hs}{H_{*}\he}
\newcommand{\pa}{\partial}

\title[Topological properties of punctual Hilbert schemes]{Topological properties of punctual Hilbert schemes of almost-complex fourfolds (I)}
\author{Julien Grivaux}
\begin{document}
\setcounter{section}{1}
\begin{abstract}
In this article, we study topological properties of Voisin's punctual Hilbert schemes of an almost-complex fourfold $X$. In this setting, we compute their Betti numbers and construct Nakajima operators. We also define tautological bundles associated with any complex bundle on $X$, which are shown to be canonical in $K$-theory.
\end{abstract}
\maketitle
\tableofcontents
\thispagestyle{empty}
\section{Introduction}
Our aim in this paper is to extend some properties of
the cohomology of
punctual Hilbert schemes on smooth projective surfaces to the case of almost-complex compact manifolds of dimension four.
\par\medskip
Let $X$ be a smooth complex projective surface. For any integer $n\in\N\ee\be$, the punctual Hilbert scheme $\xn$ is defined as the set of all $0$-dimensional subschemes of $X$ of length $n$. A theorem of Fogarty \cite{SchHilFo} states that $\xn$ is a smooth irreducible projective variety of complex dimension $2n$. The Hilbert-Chow map $\apl{HC}{\xn}{\snx}$ defined by $HC(\xi )=\sum_{x\in\supp(\xi )}\ell_{x}\he\bigl( \xi \bigr)x$ is a desingularization of the symmetric product $\snx$. This implies that the varieties $\xn$ can be seen as smooth compactifications of the sets of distinct unordered $n$-tuples of points in $X$.
\par\medskip
Voisin constructed in \cite{SchHilVo1} punctual Hilbert schemes $\xn$ when $X$ is only supposed to be a smooth almost-complex compact fourfold. This construction produces almost-complex Hilbert schemes $\xn$ which are differentiable manifolds of dimension $4n$ endowed with a stable almost-complex structure. Moreover there exists a continuous Hilbert-Chow map $\apl{HC}{\xn}{\snx}$ whose fibers are homeomorphic to the fibers of the Hilbert-Chow map in the integrable case.
\par\medskip
Using ideas of Voisin concerning relative integrable structures, we generalize to the almost-complex setting some results already known in the integrable case. In this paper, we will mainly focus on the additive structure of the cohomology groups of $\xn$ with rational coefficients. We first expose Voisin's construction and we study the local topological structure
of the Hilbert-Chow map. This allows us to compute the Betti Numbers
of $\xn$, thus extending \got's formula \cite{SchHilGo2}, \cite{SchHilGoSo} to the almost-complex case.
\par\bigskip
\noindent\textbf{Theorem 1.}
 \textit{Let $(X,J)$ be an almost-complex compact fourfold and, for any positive integer $n$, let
$\bigl( b_{i}\he(\xn)\bigr)_{i=0,\dots,4n}\he$ be the sequence of Betti numbers of the almost-complex Hilbert scheme $\xn$. Then the generating function for these Betti numbers is given by the formula }
\[
\sum_{n=0}^{\infty }\, \sum_{i=0}^{4n}b_{i}\he\bigl( \xn\bigr)t^{i}q^{n}=\prod_{m=1}^{\infty }\dfrac
{\bigl[(1+t^{2m-1}q^{m})(1+t^{2m+1}q^{m})\bigr]^{b_{1}(X)}}
{(1-t^{2m-2}q^{m})(1-t^{2m+2}q^{m})(1-t^{2m}q^{m})^{b_{2}(X)}}\cdot
\]
\par\bigskip
The proof of Theorem 1 relies on a topological version of the decomposition theorem of \cite{SchHilDBBG} for semi-small maps, which is due to Le Potier \cite{SchHilLP}.
\par\medskip
The second part of the paper is devoted to the definition and the study of the Nakajima operators $q_{i}\he(\alpha )$ of an arbitrary almost-complex compact fourfold $X$.
These operators are obtained as actions by correspondence of incidence varieties, constructed in the almost-complex setting. The incidence varieties are stratified topological spaces locally modelled on analytic spaces.
We prove in this context the Nakajima commutation relations \cite{SchHilNa}:
\par\medskip
\noindent\textbf{Theorem 2.}
\textit{For any pair $(i,j)$ of integers and any pair $(\alpha ,\beta )$ of cohomology classes in $H\ee\be(X,\Q)$ we have}
\[
\bigl[ \mathfrak{q}_{i}\he(\alpha) ,\mathfrak{q}_{j}\he(\beta )\bigr]=i\, \delta _{i+j,0}\he\, \left( \int_{X}\alpha \beta\right)\id .
\]
\par\bigskip
\noindent It follows from Theorems 1 and 2 that the Nakajima operators produce an irreducible representation of the Heisenberg super-algebra $\mathcal{H}(H^{*}\be(X,\Q))$ on
$\bop_{n\in\N}H^{*}\be(\xn,\Q)$.
\par\medskip
In the last part, we explain how to construct tautological complex bundles $E^{[n]}\be$ on the almost-complex Hilbert schemes $\xn$ starting from a complex vector bundle $E$ on $X$. To do so, we use variable holomorphic structures on $E$ in the same spirit as the variable holomorphic structures on $X$ used in Voisin's construction to define $\xn$.
\par\medskip
If $X$ is projective, Nakajima's theory as well as the tautological bundles are the fundamental tools to describe the ring structure of $H\ee(\xn,\Q)$ (see \cite{SchHilLe}). In a forthcoming paper, we will use the analogous almost-complex objects we have constructed here to compute the ring structure of $H\ee(\xn,\Q)$ when $X$ is a compact symplectic fourfold with vanishing first Betti number.
\par\bigskip
\noindent\textbf{Acknowledgement.} I want to thank my advisor Claire Voisin whose work on the almost-complex punctual Hilbert scheme is at the origin of this article. Her deep knowledge of the subject and her numerous advices have been most valuable to me. I also wish to thank her for her kindness and her patience.
\section{The Hilbert scheme of an almost-complex compact fourfold}\label{SectionUn}
\subsection{Voisin's construction}\label{SubsectionUn}
In this section, we recall Voisin's construction of the almost-complex
Hilbert scheme and establish some complementary results. Let $(X,J)$ be an
almost-complex compact fourfold. The symmetric product $\snx$ will be endowed
with the sheaf $\ci_{\snx}$ of $\mathcal{C}^{\infty}\be$  functions on $X^{n}$ invariant under
$\mathfrak{S}_{n}\he$.
Let us introduce the incidence set
\begin{equation}\label{EncoreUne}
 Z_{n}\he=\{(\xb,p)\in\snx\tim X,\ \textrm{such that}\ p\in\xb\}.
\end{equation}
\begin{definition}\label{NouvDefUn}
 For $\varepsilon >0$, let $\mathcal{B}_{\varepsilon }\he$ be the set of pairs $(W,\jre)$ such that
\begin{enumerate}
 \item [(i)] $W$ is a \nbh\ of $Z_{n}\he$ in $\snx\tim X$,
\item[(ii)] $\jre$ is a relative integrable complex structure on the fibers of $\apl{\pr_{1}}{W}{\snx}$ depending smoothly on the parameter $\xb$ in $\snx$,
\item[(iii)] if $g$ is a fixed riemannian metric on $X$, $\sup_{\xb\in\snx,\, p\in W_{\xb}\he}\he
||J^{\,\textrm{rel}}_{\xb}(p)-J_{\xb}\he(p)||\,_{g}\he\le\varepsilon $.
\end{enumerate}
\end{definition}
For $\varepsilon $ small enough, $\mathcal{B}_{\varepsilon }\he$ is connected and weakly contractible (i.e. $\pi _{i}\he(\mathcal{B}_{\varepsilon }\he)=0$ for $i\ge 1$). We choose such a $\varepsilon $ and write $\mathcal{B}$ instead of $\mathcal{B}_{\varepsilon }\he$.
\par\medskip
Let $\apl{\pi }{(W_{\rel}^{[n]},J^{\re }\be)}{\snx}$ be the relative Hilbert
scheme of $(W,\jre)$ over $\snx$. The fibers of $\pi $ are the smooth
analytic sets
$(W_{\ub{x}}^{[n]},J_{\ub{x}}^{\re })$, $\xb\in\snx$. Let
$\apl{{HC}_{\rel}\he}{W_{\rel}^{[n]}}{S_{\rel}^{n}W}$ be the relative
Hilbert-Chow morphism over $\snx$.
\begin{definition}\label{DefUnSectionUnSubsectionUn}
 The \emph{topological Hilbert scheme} $\xnj{}$ is
$(\pi ,\pr_{2}\circ HC_{\rel}\he)^{-1}\be\bigl( \Delta _{\snx}\he\bigr)$, where $\Delta
_{\snx}\he$ is the diagonal of $\snx$. More explicitely,
\[
\xnj{}=\{(\xi ,\ub{x})\ \textrm{such that}\ \ub{x}\in\snx,\
\xi \in(W_{\ub{x}}^{[n]},J_{\ub{x}}^{\re })\ \textrm{and}\ HC(\xi )=\xb\,
\}.
\]
\end{definition}
To put a differentiable structure on
$\xnj{}$, Voisin uses specific relative
integrable structures which are invariant by a compatible system of retractions
on the strata of $\snx$.
These relative structures are differentiable for a differentiable structure
$\mathfrak{D}_{\!J}\he$
on $\snx$ which depends on $J$ and is weaker than the quotient differentiable structure, i.e.
$\mathfrak{D}_{\!J}\he\suq\ci_{\snx}$.
The main result of Voisin is the following:
\begin{theorem}\cite{SchHilVo1}, \cite{SchHilVo2}\label{MainVoisin}
$\he$\par
\begin{enumerate}
\item[(i)] $\xn$ is a $4n$-dimensional differentiable manifold, well-defined modulo
diffeomorphisms homotopic to identity.
\item[(ii)] The Hilbert-Chow map $\apl{HC}{\xna{n}}{\bigl( \snx,\mathfrak{D}_{\!J}\he\bigr)}$ is differentiable and its fibers $HC^{-1}\be(\xb)$ are homeomorphic to the fibers of the usual Hilbert-Chow morphism for any integrable structure near $\xb$.
\item[(iii)] $\xn$ can be endowed with a stable almost-complex structure, hence defines
an almost-complex cobordism class. When $X$ is symplectic, $\xnj{}$ is symplectic.
\end{enumerate}
\end{theorem}
The first point is the analogue of Fogarty's result \cite{SchHilFo} in the
differentiable case. In this article we will not use differentiable properties of
$X^{[n]}\be$ but only topological ones, which allows us to work with $\xnj{}$ for any $J^{\, \rel}\be$ in $\mathcal{B}$. Without any assumption on $J^{\, \rel}\be$, the point (i) in Theorem \ref{MainVoisin} has the following topological version:
\begin{proposition}\label{PropUnSectionUnSubsectionUn}
If $J^{\re }\be\!\! \in\mathcal{B}$, $\xnj{}$ is a $4n$-dimensional
topological manifold.
\end{proposition}
\begin{proof}
 Let $\xb_{\, 0}\in\snx$. There exist holomorphic relative coordinates
$(z_{\xb}\he,
w_{\xb}\he)$ for $J^{\re }_{\xb}$ in a neighbourhood of $\xb_{\, 0}$ which depend
smoothly on $\xb$. For every $\xb$ near $\xb_{\, 0}$, the map
$\flgd{p}{(z_{\xb}(p),w_{\xb}(p))}$ is a biholomorphism between
$(W_{\xb}\he,J^{\re }_{\xb})$ and its image in
$\C^{2}\be$ with the standard complex structure. Let us write
$(z_{\xb}(p),w_{\xb}(p))=
\phi (x,p)$, where $\phi $ is a smooth function defined for $x$ near a lift $x_{0}\he $ of
$\xb_{\, 0}$, invariant under the action of the stabilizer of
$x_{0}\he $ in
$\mathfrak{S}_{n}$.
\par
We write
 $x_{0}\he =(y_{1},\dots,y_{1},\dots,y_{k},\dots,y_{k})$ where the points $y_{j}$
are pairwise distinct and each $y_{j}$ appears $n_{j}$ times.
We will identify small distinct \nbh s of $y_{j}$ in $X$ with distinct balls
$B(y_{j},\varepsilon )$ in $\C^{2}\be$. $\phi $ is defined on $B(y_{1},\varepsilon
)^{n_{1}}\be\times\cdots\times B(y_{k},\varepsilon
)^{n_{k}}\be\times\cup_{j=1}^{k}
B(y_{j},\varepsilon )$ and is $\mathfrak{S}_{n_{1}}\times\cdots\times
\mathfrak{S}_{n_{k}}$ invariant.
We can also suppose that $\phi (x_{0}\he ,\, .\, )=\textrm{id} $. We introduce new holomorphic coordinates by the formula
$\ti{\phi }(x,p)=\phi (x,p)-D_{1}\, \phi (x_{0}\he ,y_{j})(x-x_{0}\he)$ if $p\in B(y_{j},\varepsilon )$. Let
\[
\apl{\Gamma }{B(y_{1},\varepsilon
)^{n_{1}}\be\times\cdots\times B(y_{k},\varepsilon
)^{n_{k}}\be}{\bigl( \C^{2}\bigr)^{n}}
\] be defined by
\[
\Gamma (x_{1},\dots,x_{n})=(\ti{\phi }(x_{1},\dots,x_{n},x_{1}),\dots,\ti{\phi }(x_{1},\dots,x_{n},x_{n})).
\]
The map
$\Gamma $ is $\mathfrak{S}_{n_{1}}\times\cdots\times
\mathfrak{S}_{n_{k}}$ equivariant and has for differential at $x_{0}\he $ the identity map, so it induces a local homeomorphism $\gamma $ of $\snx$ around $\xb_{\, 0}$. The image of the chart of $(\xn,J^{\re }\be)$ over a neighbourhood of $\xb_{\, 0}$ will be the classical Hilbert scheme
$(\C^{2})^{[n]}\be$ over a neighbourhood of $\xb_{\, 0}$. The chart and its inverse are given by the formulae $\varphi (\xi )=\phi (\xb,\, .\, )_{*}\he$, where $ HC(\xi )=\xb$, and
$\varphi ^{-1}(\eta )=(\phi (\ub{y},\, .\, )\be^{-1})_{*}\he\eta$, where $\ub{y}=\gamma ^{-1}
(HC(\eta ))$.
\end{proof}
\begin{remark}
 Let $J_{0}^{\re }$ and $J_{1}^{\re }$ be two relative integrable complex structures, and let
$\phi _{0}$, $\phi _{1}$, $\gamma _{0}$ and $\gamma _{1}$ be defined as above. Then
$\xnj{0}$ and $\xnj{1}$ are homeomorphic over a \nbh\ of $\xb_{\, 0}$. If $\phi (\xb,p)=\phi _{1}^{-1}(\gamma _{1}^{-1}\gamma _{0}(\xb),\phi_{0}\he (\xb,p))$ and $\gamma (\xb)=\gamma _{1}^{-1}\gamma _{0}(\xb)$, then there is a commutative diagram
\[
\xymatrix{
\xnj{1}\ar[d]^{HC}&HC^{-1}(V_{\xb_{\, 0}})\ar@{_{(}->}[l]\ar[r]^{\phi _{*}\he}_{\sim}\ar[d]&\,
HC^{-1}(\ti{V}_{\xb_{\, 0}})\, \ar[d]\ar@{^{(}->}[r]&\xnj{1}\ar[d]^{HC}\\
\snx\ar@{}@<-1.5ex>[r]^{\supseteq}&V_{\xb_{\, 0}}\ar[r]^{\gamma }_{\sim}&\ti{V}_{\xb_{\, 0}}\he
\ar@{}@<-1.5ex>[r]^{\subseteq}&\snx
}
\]
and $\gamma $ is a stratified isomorphism.
\end{remark}
\subsection{\got's formula}
We will now turn our attention to the cohomology of $\xnj{}$. The first step is the computation of the Betti numbers of $\xn$. We first recall the proof of G\"{o}ttsche and Soergel (\cite{SchHilGoSo}) and then we show how to adapt it in the non-integrable case.
\par\medskip
Let $X$ and $Y$ be irreducible algebraic complex varieties and $\apl{f}{Y}{X}$ be a proper morphism. We assume that $X$ is stratified in such a way that $f$ is a topological fibration over each stratum $X_{\nu }\he$. We denote by $d_{\nu }\he$ the real dimension of the largest irreducible component of the fiber. If $Y_{\nu }\he=f^{-1}(X_{\nu }\he)$, $\LL_{\nu }\he=R^{\, d\nu }\be f_{*}\he\Q_{Y_{\nu }\he}$ will be the associated monodromy local system on $X_{\nu }\he$.
\begin{definition}\label{DefUnSectionUnSubsectionDeux}
 -- The map $f$ is a \emph{semi-small morphism} if for all $\nu $, $\codim_{X}\he X_{\nu }\ge d_{\nu }\he$.
\par
-- A stratum $X_{\nu }\he$ is \emph{essential} if $\codim_{X}\he X_{\nu }\he= d_{\nu }\he$.
\end{definition}
\begin{theorem}\label{ThUnSectionUnSubsectionDeux}\cite{SchHilDBBG}
 If $Y$ is rationally smooth and $\apl{f}{Y}{X}$ is a proper semi-small morphism, there exists a canonical quasi-isomorphism
$\xymatrix{
Rf_{*}\he\Q_{Y}\ar[r]^(.28){\sim}&\mkern -13mu\bigoplus\limits_{\nu \ \emph{essential}}j_{\nu *}\he\, IC_{\ba{X}_{\nu }\he}^{\, ^{\bullet}}(\LL_{\nu }\he)[-d_{\nu }\he]
}$
in the bounded derived category of $\Q$-constructible sheaves on $X$,
where $IC_{\ba{X}_{\nu }\he}^{\, ^{\bullet}}(\LL_{\nu }\he)$ is the intersection complex on $\ba{X}_{\nu }\he$ associated to the monodromy local system $\LL_{\nu }\he$ and $\apl{j_{\nu }\he}{\ba{X}_{\nu }\he}{X}$ is the inclusion. In particular,
$H^{k}(Y,\Q)=\bigoplus\limits_{\nu \ \emph{essential}}IH^{k-d_{\nu }\he}\be(\ba{X}_{\nu }\he,\LL_{\nu }\he)$.
\end{theorem}
\begin{remark}\label{RemarkAbove}
A simple proof of Theorem \ref{ThUnSectionUnSubsectionDeux} is done in \cite{SchHilLP} and can be found in \cite{SchHilGri}. Furthermore, this proof shows that
$Rf_{*}\he\Q_{Y}\simeq \bigoplus\limits_{\nu \ \textrm{essential}}j_{\nu *}\he\, IC_{\ba{X}_{\nu }\he}^{\, ^{\bullet}}(\LL_{\nu} \he)[-d_{\nu }\he]$ under the following weaker topological hypotheses: $Y$
is a rationally smooth topological space,
$X$ is a stratified topological space and
$\apl{f}{Y}{X}$ is a proper map which is locally equivalent over $X$ to a semi-small map between complex analytic varieties.
\end{remark}
If $X$ is a quasi-projective surface, the Hilbert-Chow morphism is semi-small with irreducible fibers (see \cite{SchHilBr}), so that the monodromy local systems are trivial, and $\xn$ is smooth. The decomposition theorem gives G\"{o}ttsche's formula for $b_{i}(\xn)$. We now show that \got's formula holds more generally for almost-complex Hilbert schemes.
\begin{theorem}[G\"{o}ttsche's formula]\label{ThDeuxSectionUnSubsectionDeux}
 If $(X,J)$ is an almost-complex compact fourfold, then for any integrable complex structure
$\jre$,
\[
\sum_{n=0}^{\infty }\, \sum_{i=0}^{4n}b_{i}\he \bigl(\xnj{}\bigr)t^{i}\be q^{n}=\prod_{m=1}^{\infty }\dfrac
{\bigl[(1+t^{2m-1}q^{m})(1+t^{2m+1}q^{m})\bigr]^{b_{1}(X)}}
{(1-t^{2m-2}q^{m})(1-t^{2m+2}q^{m})(1-t^{2m}q^{m})^{b_{2}(X)}}\cdot
\]
\end{theorem}
\begin{proof}
 By Remark \ref{RemarkAbove}, it suffices to check that
$\apl{HC}{\xnj{}}{\snx}$ is locally equivalent to a semi-small morphism. The proof of
Proposition \ref{PropUnSectionUnSubsectionUn}
shows that $\apl{HC}{\xnj{}}{\snx}$ is locally equivalent to $\apl{HC}{U^{[n]}\be}{S^{[n]}\be U}$ where $U$ is an open set in $\C^{2}\be$. Thus the decomposition theorem applies and the computations are the same as in the integrable case.
\end{proof}
\subsection{Variation of the relative integrable structure}
Theorem \ref{ThDeuxSectionUnSubsectionDeux} implies in particular that the Betti numbers of $\xnj{}$ are independent of $\jre$. We now prove a stronger result, namely that the Hilbert schemes corresponding to different relative integrable complex structures are homeomorphic.
\begin{proposition}\label{PropDeuxSectionUnsubsectionDeux}
$\he$\par
\begin{enumerate}\label{NouvPropUn}
 \item [(i)] Let $\bigl( J_{t}^{\re }\bigr)_{t\in B(0,r)\suq\R^{d}\be}\he$ be a smooth path in $\mathcal{B}$. Then the associated relative Hilbert scheme $\bigl( \xn,\bigl\{J_{t}^{\re }\bigr\}_{t\in B(0,r)}\he\bigr)$ over
$B(0,r)$ is a topological fibration.
\item[(ii)] If $J_{0}^{\re }$, $J_{1}^{\re }\in\mathcal{B}$, then
there exist canonical isomorphisms
$\hb \bigl(\xnj{0},\Q\bigr)\simeq\hb \bigl(\xnj{1},\Q\bigr)$ and $K\bigl(\xnj{0}\bigr)\simeq K\bigl(\xnj{1}\bigr)$.
\end{enumerate}
\end{proposition}
In order to prove Proposition \ref{NouvPropUn}, we first establish the following result:
\begin{proposition}\label{PropUnSectionUnsubsectionDeux}
Let $\bigl( J_{t}^{\re }\bigr)_{t\in B(0,r)\suq\R^{d}}\he$ be a family of smooth relative complex structures in a \nbh\ of $Z_{n}\he$ varying smoothly with $t$. Then there exist
$\varepsilon >0$, a \nbh\ $W$ of $Z_{n}\he$ in $\snx\tim X$ and a continuous map
$\xymatrix@C=17pt{
\psi :(t,\xb,p)\ar@{|->}[r]&\psi _{t,\xb}\he(p)
}$
from $B(0,\varepsilon )\tim W$ to $X$
such that:
\begin{enumerate}
 \item [(i)] $\psi _{0,\xb}\he(p)=p$,

 \item [(ii)] For fixed $(t,\xb)$, $\psi _{t,\xb}\he$ is a biholomorphism between a \nbh\ of $\xb$ and a \nbh\ of $\sn\psi _{t,\xb}\he(\xb)$, endowed with the structures $J^{\re }_{0,\xb}$ and
$J^{\re }_{t,\psi _{t,\xb}\he(\xb)}$,
 \item [(iii)] For all $t$ in $B(0,\varepsilon )$, the map $\flgd{\xb}{\sn\psi _{t,\xb}\he(\xb)}$ is a homeomorphism of $\snx$.
\end{enumerate}
\end{proposition}
\begin{proof}
We can choose a family of maps $\theta _{t}\he$ varying smoothly with $t$ such that for all $\xb$ in
$\snx$ and $t$ in $B(0,r)$, $\theta _{t,\xb}\he$ is a biholomorphism between two \nbh s of $\xb$ endowed with the structures $J^{\re }_{t,\xb}$ and $J^{\re }_{0,\xb}$, and such that for all $\xb$ in $ \snx$, $\theta _{0,\xb}\he=\id$. We take, as in the proof of Proposition
\ref{PropUnSectionUnSubsectionUn}, a system $\bigl( \phi ^{i}_{\xb}\bigr)_{1\le i\le N}$ of holomorphic relative coordinates for $J_{0}^{\re }$ with respect to a covering
$\bigl(\ti{U}_{i}\he \bigr)_{1\le i\le N}$ of $\snx$ such that
$\flgd{\xb}{\sn \phi ^{i}_{\xb}(\xb)}$ is a homeomorphism between $\ti{U}_{i}\he$ and its image $\ti{V}_{i}\he$ in $\sn\, \C^{2}$. We define holomorphic relative coordinates
$\bigl( \phi ^{i}_{t,\xb}\bigr)_{1\le i\le N}$ for $J_{t}^{\rel}$ by the formula
$ \phi ^{i}_{t,\xb}(p)= \phi ^{i}_{\xb}\bigl( \theta _{t,\xb}\he(p)\bigr)$. For small $t$, after shrinking $\ti{U}_{i}\he$ if necessary,
$\flgd{\xb}{\sn \phi ^{i}_{t,\xb}(\xb)}$ is still a homeomorphism: indeed the map $\flgd{\xb}{\sn\phi _{t,\xb}^{i}(\xb)}$ is obtained from the $\mathfrak{S}_{n_{1}}\he\times\cdots \times \mathfrak{S}_{n_{k}}\he$ equivariant smooth map
$
\flgd{(x_{1},\dots ,x_{n})}{\bigl( \phi _{t}\he(x_{1},\dots ,x_{n},x_{1}),\dots ,\phi _{t}\he(x_{1},\dots ,x_{n},x_{n})\bigr)}.
$
Then we use the fact that a sufficiently small smooth perturbation of a smooth diffeomorphism  remains a smooth diffeomorphism.
\par\medskip
Let $\ti{Z}_{n}\he\suq \sn\, \C^{2}\times\C^{2}$ be the incidence variety of $\C^{2}\be$. The map $\apl{\check{\phi }_{t}^{i}}
{(\xb,p)}{\bigl( \sn\phi ^{i}_{t,\xb}(\xb),\phi ^{i}_{t,\xb}(p)\bigr)}$ is a homeomorphism between two \nbh s of $Z_{n}\he$ and $\ti{Z}_{n}\he$ over $\ti{U}_{i}\he$ and $\ti{V}_{i}\he$.
If we define $\apl{\check{\psi }_{t}\he}{(\xb,p)}{\bigl( \sn\psi _{t,\xb}\he(\xb),\psi _{t,\xb}\he(p)\bigr)}$\!\!, the condition (ii) of the proposition means that
$\check{\phi }_{t}^{i}\circ\check{\psi }_{t}\he\circ\bigl( \check{\phi }_{0}^{i}\bigr)^{-1}$ is of the form $\flgd{(\ub{y},p)}{\bigl( \sn u_{t,\ub{y}}\he(\yb),u_{t,\ub{y}}\he(p)\bigr)}$ where $\ub{y}\in\ti{V}_{i}\he$ and $u_{t,\ub{y}}\he$ is a biholomorphism between a \nbh\ of $\yb$ and its image (both endowed with the standard complex structure of $\C^{2}\be$), varying smoothly with $t$ and $\yb$. The condition (i) means that $u_{0,\ub{y}}\he=\id$. Thus $\bigl( \psi _{t}\he\bigr)_{||t||\le\varepsilon }$ can be constructed on small open sets of $\snx$. Since biholomorphisms close to identity form a contractible set, we can, using cut-off functions, glue together the local solutions on $\snx$ to obtain a global one. The map $\flgd{\xb}{\sn\psi _{t,\xb}\he(\xb)}$ is induced by a smooth $\mathfrak{S}_{n}$-\,equivariant map of $X^{n}\be$ into $X^{n}$ (and is a small perturbation of the identity map if $||t|| $ is small enough), thus a $\mathfrak{S}_{n}$-\,equivariant diffeomorphism of $X^{n}$. We have therefore defined a family of maps $\bigl( \psi _{t}\he\bigr)_{||t||\le\varepsilon }\he$ satisfying the conditions (i), (ii) and (iii).
\end{proof}
We can now prove Proposition \ref{NouvPropUn}.
\par\medskip
\noindent\textit{Proof of Proposition \ref{NouvPropUn}.}
(i)
We have
\[
\bigl( \xn,\bigl\{J_{t}^{\re }\bigr\}_{t\in B(0,r)}\he\bigr)\!=\!\Bigl\{
(\xi ,\xb,t)\ \textrm{such that}\ \xb\in\snx,\ t\in B(0,r),\ \xi \in \bigl( W_{\xb}^{[n]},J_{t,\xb}^{\re }\bigr),\ HC(\xi )=\xb
\Bigr\}\cdot
\]
A topological trivialization of this family over $B(0,r)$ near zero is given by the map
\[
\apl{\Gamma }{\xnj{0}\times B(0,\varepsilon )}{\bigl( \xn,
\bigl\{J_{t}^{\re }\bigr\}_{t\in B(0,\varepsilon )}\bigr)}
\]
defined by
$\Gamma (\xi ,\xb,t)=\bigl( \psi _{t,\xb*}\he\xi ,\psi _{t,\xb}\he(\xb),t\bigr)$, where $\psi $ is given by Proposition \ref{PropUnSectionUnsubsectionDeux}. This proves that the relative Hilbert scheme is locally topologically trivial over $B(0,r)$.
\par\medskip
(ii) The set $\mathcal{B}$ being connected, point (i) shows that $\xnj{0}$ and
$\xnj{1}$ are homeomorphic. Since $\pi _{1}(\mathcal{B})=0$, if we consider two paths
$\bigl( J_{0,t}^{\re }\bigr)_{0\le t\le 1}\he$ and $\bigl( J_{1,t}^{\re }\bigr)_{0\le t\le 1}\he$ between $J_{0}^{\re }$ and $J_{1}^{\re }$, we can find a smooth family $\bigl( J_{s,t}^{\re }\bigr)_{\genfrac{}{}{0pt}{}{0\le s \le 1}{0\le t \le 1}}\he$ which is an homotopy between the two paths. The relative associated Hilbert scheme over $[0,1]\times [0,1]$ is locally topologically trivial, hence globally trivial since $[0,1]\tim[0,1]$ is contractible. This shows that the homeomorphisms between
$\xnj{0}$ and $\xnj{1}$ constructed by the procedure above belong to a canonical homotopy class.\hfill$\square$
\par\medskip
As a consequence of this proposition, there exists a ring $H\ee\be(\xn,\Q)$ (resp. $K( \xn)$) such that for any $J^{\re }\be$ close to $J$, $H\ee\be(\xn,\Q)$ (resp. $K( \xn)$) and $H\ee\be\bigl(\xnj{}, \Q\bigr)$ (resp. $K\bigl( \xnj{}\bigr)$) are canonically isomorphic.
\par\medskip
In the sequel, we will deal with products of Hilbert schemes. We can of course consider products of the type $\bigl( \xnj{n}\bigr)\tim\bigl( X^{[m]}_{J^{\re}_{m}}\bigr)$, but in practice it is necessary to work with pairs of relative integrable complex structures parametrized by $(\xb,\yb)$ in $\snx\tim S^{m}_{\be}X$. Let us introduce the incidence set
\begin{equation}\label{EqUn}
 Z_{n\tim m}\he=\bigl\{(\xb,\yb,p)\ \textrm{in}\ \bigl( \snx\tim S^{m}_{\be}X\bigr)\tim X\ \textrm{such that}\ p\in\xb\cup\yb\bigr\}\cdot
\end{equation}
Let $W$ be a small \nbh\ of
$Z_{n\times m}\he$ and let $J^{1,\rel}\be$ and $J^{2,\rel}\be$ be two relative integrable complex structures on the fibers of $\apl{\pr_{1}\he}{W}{\snx\times S^{m}\be X}$.
\begin{definition}\label{DefUnSectionUnSubsectionTrois}
The product Hilbert scheme $\bigl( X^{[n]\times [m]}\be,J^{1,\rel}\be,J^{2,\rel}\be\bigr)$ is defined by
\begin{align*}
\bigl( X^{[n]\times [m]}\be,J^{1,\rel}\be,J^{2,\rel}\be\bigr)=
\Bigl\{(\xi ,\eta ,\xb,\yb)\ \textrm{such that}\ &\xb\in\snx,\ \yb\in S^{m}\be X,\
\xi \in\bigl( W^{[n]}_{\xb,\yb},J^{1,\rel}_{\xb,\yb}\bigr),\\
&\eta \in\bigl( W^{[m]}_{\xb,\yb},J^{2,\rel}_{\xb,\yb}\bigr),\
HC(\xi )=\xb,\ HC(\eta )=\yb\Bigr\}\cdot
\end{align*}
The same definition holds for products of the type $\Bigl( X^{[n_{1}\he]\tim\cdots\tim [n_{k}\he]}\be,J^{\, 1,\, \rel}\be,\dots, J^{\, k,\, \rel}\be\Bigr)$.
\end{definition}
If there exist two relative integrable complex structures $J_{n}^{\, \rel}$ and $J_{m}^{\, \rel}$ in \nbh s of $Z_{n}\he$ and $Z_{m}\he$ such that $J^{\, 1,\rel}_{\xb,\yb}=J^{\re }_{n,\xb}$ and $J^{\, 2,\rel}_{\xb,\yb}=J^{\re }_{m,\xb}$
in small \nbh s of $\xb$ and $\yb$, we have
\[
\bigl( X^{[n]\times [m]}\be ,J^{\, 1,\rel}\be,
J^{\, 2,\rel}\be \bigr)=\xnj{n} \times X^{[m]}_{J^{\re}_{m}}.
\]
If $\bigl( J^{1,\rel}_{t},J^{2,\rel}_{t}\bigr)_{t\in B(0,r)}\he$ is a smooth family of relative integrable complex structures, it can be shown as in Propositions \ref{PropUnSectionUnsubsectionDeux} and \ref{PropDeuxSectionUnsubsectionDeux} that the family
$\Bigl( X^{[n]\times [m]}\be, \bigl\{J^{1,\rel}_{t}\bigr\}_{t\in B(0,r)}\he,\bigl\{J^{2,\rel}_{t}\bigr\}_{t\in B(0,r)}\he\Bigr)$ is topologically trivial over
$B(0,r)$. Thus the product Hilbert schemes $\bigl( X^{[n]\tim[m]}\be,J^{1,\rel}\be,J^{2,\rel}\be\bigr)$ is isomorphic to products
$\xnj{n}\tim X^{[m]}_{J^{\re}_{m}}$ of usual Hilbert schemes. If the structures $J^{1,\rel}\be$ and $J^{2,\rel}\be$ are equal,
$\bigl( X^{[n]\tim[m]}\be,J^{1,\rel}\be,J^{2,\rel}\be\bigr)$ consists of pairs of schemes of given support (parametrized by $\snx\tim S^{m}\be X$) for the \emph{same} integrable structure. These product Hilbert schemes are therefore well adapted for the study of incidence relations.

 \section{Incidence varieties and Nakajima operators}\label{SecDeux}
\subsection{Construction of incidence varieties}
If $J$ is an integrable complex structure on $X$, the \emph{incidence variety} $X^{\npn}\be$ is classically defined by
\[
X^{\npn}\be=\bigl\{(\xi, \xi')\ \textrm{such that}\ \xi \in\xn,\ \xi '\in X^{[n']}\be\ \textrm{and}\ \xi \suq\xi '\bigr\}\cdot
\]
The incidence variety
$X^{\npn}\be$ is never smooth unless $n'=n+1$ (see \cite{SchHilTi}). We have three maps
$\apl{\lambda }{X^{\npn}\be}{\xn}$\!\!, $\apl{\nu}{X^{\npn}\be}{X^{[n']}\be}$ and
$\apl{\rho }{X^{\npn}\be}{S^{n'-n}\be X}$ given by $\lambda (\xi ,\xi ')=\xi $,
$\nu (\xi ,\xi ')=\xi '$ and $\rho (\xi ,\xi ')=\textrm{supp}\bigl(\ii_{\xi }\he/_{\ds \ii_{\xi '}\he}\bigr)$. Note that, by definition, $(\lambda ,\nu )$ is injective.
\par\medskip
If $J$ is not integrable, we can define $X^{\npn}\be $ using the relative construction.
Let $J^{\re }_{n\times(n'-n)}$ be a relative integrable complex structure in a \nbh\ of
$Z_{n\times(n'-n)}\he$ in $\snx\tim S^{n'-n}\be X\tim X$ (see (\ref{EqUn})).
\begin{definition}\label{DefUnSectionDeuxSubsectionUn}
The incidence variety $\bigl( X^{\npn},J_{n\tim(n'-n)}^{\rel}\bigr)$ is defined by
\begin{align*}
\bigl(X^{\npn}\be, J_{n\tim(n'-n)}^{\re } \bigr)=\Bigl\{
(\xi ,\xi ',\xb,\yb)&\ \textrm{such that}\ \xb\in\snx,\ \yb\in S^{n'-n}\be X,\
\xi \in\bigl( W_{\xb}^{[n]},J_{n\tim(n'-n),\, \xb,\,  \yb}^{\re }\bigr),\\
&\xi' \in\bigl( W_{\xb\, \cup\, \yb}^{[n']},J_{n\tim(n'-n),\, \xb,\, \yb}^{\re }\bigr),\ \xi \suq\xi ',\ HC(\xi )=\xb,\ \rho (\xi ,\xi ')=\yb
\Bigr\}\cdot
\end{align*}
\end{definition}
Let $J^{\re }_{n\times n'}$ be a relative integrable complex structure in a \nbh\ of
$Z_{n\tim n'}\he$ such that for every
$\ub{u}\in \snx$ and $\ub{v}\in S^{n'-n}\be X$,
$J^{\re }_{n\times n',\ub{u},\ub{u}\, \cup\, \ub{v}}=
J^{\re }_{n\times (n'-n),\ub{u},\ub{v}}$. Then
\[
\bigl( X^{\npn},J^{\re }_{n\times (n'-n)}\bigr)\suq \bigl( X^{[n]\times [n']}\be ,
J^{\re }_{n\times n'},J^{\re }_{n\times n'}\bigr).
\]
If
$\bigl\{J^{\re }_{t,n\times n'}\bigr\}_{t\in B(0,r)}\he$ is a smooth family of relative complex structures, we can take, as in Proposition \ref{PropDeuxSectionUnsubsectionDeux}, a topological trivialization of
$\bigl( X^{[n]\times [n']}\be ,
\bigl\{J^{\re }_{t,n\times n'}\bigr\}_{t\in B(0,r)}\he,\bigl\{J^{\re }_{t,n\times n'}\bigr\}_{t\in B(0,r)}\he\bigr)$.
If we define $J^{\re }_{t,n\times (n'-n)}$ in a \nbh\ of $Z_{n\tim(n'-n)}\he$ by the formula
$J^{\re }_{t,n\times (n'-n),\ub{u},\ub{v}}=J^{\re }_{t,\, n\times n',\ub{u},\, \ub{u}\, \cup\, \ub{v}}$, then
we can choose the trivialization so that
the subfamily
$\bigl( X^{\npn}\be ,\bigl\{J^{\re }_{t,n\times (n'-n)}\bigr\}_{t\in B(0,r)}\he\bigr)$ is sent to the product
$U^{\npn}\be\tim B(0,\varepsilon )$, where $U$ is an open set of $\C^{2}\be$. This means that the fami\-ly
\[
\Bigl\{ \Bigl(X^{\npn}\be,\bigl\{J^{\re}_{n\tim(n'-n)}\bigr\}_{t\in B(0,r)}\he\Bigr), \Bigl(X^{[n]\times [n']}\be,\bigl\{J^{\re}_{n\tim n'}\bigr\}_{t\in B(0,r)}\he,\bigl\{J^{\re}_{n\tim n'}\bigr\}_{t\in B(0,r)}\he\Bigr)\Bigr\}
\]
is locally, hence globally topologically trivial over $B(0,r)$.
\par\medskip
The natural morphism
from $\bigl(X^{\npn}\be ,J_{n\tim(n'-n)}\bigr)$ to $\snx\times S^{n'-n}\be X$
is locally equivalent over $\snx\tim S^{n'-n}\be X$
to the natural morphism $\flgd{U^{\npn}\be }{S^{n}\be U\tim S^{n'-n}\be U.}$ This enables us to define a stratification on $X^{\npn}\be$ by patching together the analytic stratifications of a collection of $U^{\npn}_{i}$. In this way, $\bigl(X^{\npn}\be ,J_{n\tim(n'-n)}^{\re }\bigr)$ becomes a stratified \mbox{$CW$-complex} such that for each stratum $S$, $\dim(\overline{S\he}\backslash S)\le \dim S-2$. In particular, each stratum has a homology class.
\par\medskip
\noindent Let us introduce the following notations:
\begin{enumerate}
 \item [(i)] The inverse image of the small diagonal of $\snx$ by $\apl{HC}{\xnj{n}}
{\snx}$ will be denoted by $\bigl( X^{[n]}_{0},J_{n}^{\re }\bigr)$.
\item [(ii)] The inverse image of the small diagonal of $S^{n'-n}\be X$ by
$\apl{\rho }{\bigl( X^{\npn}\be,J_{n\times (n'-n)}^{\re }\bigr)}{S^{n'-n}\be X}$ will be denoted by
$\bigl( X^{\npn}_{0},J_{n\tim(n'-n)}^{\re }\bigr)$.
\end{enumerate}
In the integrable case, $X^{\npn}_{0}$ is stratified by analytic sets $\bigl( Z_{l}\he\bigr)_{l\ge 0}\he$ defined by
\begin{equation}\label{EqDeux}
 Z_{l}\he=\bigl\{(\xi ,\xi ')\in X^{\npn}_{0}\ \textrm{such that if}\ x=\rho (\xi ,\xi '), \ \ell_{x} (\xi )=l\bigr\};
\end{equation}
$Z_{0}\he $ is irreducible of complex dimension $n'+n+1$, and all the other $Z_{l}\he$ have smaller dimensions (see \cite{SchHilLe}). By the same argument as above, this stratification also exists in the almost complex case. We prove the topological irreducibility of $Z_{0}\he $ in the following lemma:
\begin{lemma}\label{LemUnSectionDeuxSubsectionUn}
 Let $\bigl[\, \ba {Z\vphantom{^{'}}}_{0}\he\bigr]$ be the fundamental homology class of $\ba {Z\vphantom{^{'}}}_{0}\he$. Then
\[
H_{2(n'+n+1)}\he\bigl( X_{0}^{\npn},J_{n\tim(n'-n)}^{\re },\Z\bigr)=\Z\, .\bigl[\, \ba {Z\vphantom{^{'}}}_{0}\he\bigr].
\]
\end{lemma}
\begin{proof}
 It is enough to prove that the Borel-Moore homology group
$H^{\textrm{lf}}_{2(n'+n+1)}(Z_{0}\he,\Z)$ is $\Z$, since all the remaining strata
$\bigl( Z_{l}\he\bigr)_{l\ge 1}\he$ have dimensions smaller than $2(n'+n+1)-2$. Let
\begin{align*}
 \ti{Z}_{0}\he=\Bigl\{(\xi ,\eta ,\xb,p)\ \textrm{such that}\ \xb\in \snx&,\ p\in X,\
\xi \in\bigl( W_{\xb,\, (n'-n)p}^{[n]},J^{\re }_{n\times (n'-n),\,\xb,\,  (n'-n)p}\bigr),\ HC(\xi )=\xb,\\
&\eta \in \bigl( W_{\xb,\, (n'-n)p}^{[n'-n]},J^{\re }_{n\times (n'-n),\, \xb,\,  (n'-n)p}\bigr),\ HC(\eta )=(n'-n )p\Bigr\}\cdot
\end{align*}
There is a natural inclusion $\xymatrix@C=10pt{Z_{0}\he\ar@{^{(}->}[r]&\ti{Z}_{0}\he}$ given by
$\flgd{\bigl(\xi ,\xi ',\xb,(n'-n)p\bigr)}{(\xi ,\xi '_{\vert p},\xb,p)}$. Remark that $\ti{Z}_{0}\he$ is compact.
Since $\dim (\ti{Z}_{0}\he\backslash Z_{0}\he)\leq 4n+2(n'-n-1)=2(n'+n-1)$, it suffices to show that
$H_{2(n'+n+1)}\he(\ti{Z}_{0}\he,\Z)=\Z$. $\ti{Z}_{0}\he$ is a product-type Hilbert scheme homeomorphic to
$\xnj{n}\times\bigl( X_{0}^{[n'-n]},J^{\re }_{n'-n}\bigr)$.
Since
$\bigl( X_{0}^{[n'-n]},J^{\re }_{n'-n}\bigr)$ is by Brian\c{c}on's theorem \cite{SchHilBr} a topological fibration on $X$ whose fiber is homeomorphic to
an irreducible algebraic variety of complex dimension $n'-n-1$,  we obtain the result.
\end{proof}
\subsection{Nakajima operators}
We are now going to construct Nakajima operators $q_{n}\he(\alpha )$ in the almost-complex context.
If $n'>n$, let us define
\begin{equation}\label{EqTrois}
 Q^{\npn}\be=\ba{Z\vphantom{^{'}}}_{0}\he\suq \bigl(X^{[n]\times[n']}\be\times X,{J}_{n\tim n'}^{\re },J_{n\tim n'}^{\re }\bigr)\times X,
\end{equation}
where the map on the last coordinate is given by the relative residual morphism and $Z_{0}\he$ is defined by (\ref{EqDeux}). Since the pair
$\bigl(Q^{\npn}\be,\xnb{[n]\tim[n']}\tim X \bigr)$ is topologically trivial when $J_{n\tim n'}^{\re } $ varies, the cycle class $\bigl[ Q^{\npn}\be\bigr]\in H_{2(n'+n+1)}\he
\bigl( \xn\tim X^{[n']}\be \times X,\Z\bigr)$ is independent of  $J_{n\tim n'}^{\re } $.

\begin{definition}\label{DefUnSectionDeuxSubsectionDeux}
 Let $\alpha \in \hb(X,\Q)$ and $j\in\N\ee\be$. We define the Nakajima operators $\mathfrak{q}_{j}\he(\alpha )$ and
$\mathfrak{q}_{-j}\he(\alpha )$ as follows:
\[
\begin{array}{rcccl}
 \mathfrak{q}_{j}\he(\alpha )\mkern -10 mu&:&\mkern -15 mu\bop_{n\in\N}\hb (\xn,\Q)&\xymatrix{\ar[r]&}&\bop_{n\in\N}\hb(X^{[n+j]}\be,\Q)\\
&&\tau &\xymatrix{\ar@{|->}[r]&}&PD^{-1}\be\Bigl[ \pr_{2*}\he\bigl( \bigl[ Q^{[n+j,n]}\bigr]\cap(\pr_{3}\ee\alpha \cup \pr_{1}\ee\tau )\bigr)\Bigr]\\[4ex]
\mathfrak{q}_{-j}\he(\alpha )\mkern -10 mu&:&\bop_{n\in\N}\hb (X^{[n+j]}\be,\Q)&\xymatrix{\ar[r]&}&\bop_{n\in\N}\hb(X^{[n]}\be,\Q)\\
&&\tau &\xymatrix{\ar@{|->}[r]&}&PD^{-1}\be\Bigl[ \pr_{1*}\he\bigl( \bigl[ Q^{[n+j,n]}\bigr]\cap(\pr_{3}\ee\alpha \cup \pr_{2}\ee\tau )\bigr)\Bigr]
\end{array}
\]
where $\pr_{1}\he$, $\pr_{2}\he$ and $\pr_{3}\he$ are the projections from $\xn\times X^{[n+j]}\be \times X$ to each factor and $PD$ is the Poincar\'{e} duality isomorphism between cohomology and homology. We also set $\mathfrak{q}_{0}\he(\alpha )=0$
\end{definition}
\begin{remark}\label{RemUnSectionDeuxSubsectionDeux}
 Let  $|\alpha |$ be the degree of $\alpha $, then $q_{j}\he(\alpha )$ maps
$H^{i}\be(\xn,\Q)$ to $H^{i+|\alpha |+2j-2}\be(X^{[n+j]}\be,\Q)$.
\end{remark}
We now prove the following extension to the almost-complex case of Nakajima's theorem \cite{SchHilNa}:
\begin{theorem}\label{PropUnSectionDeuxSubsectionDeux}
 For $i$, $j\in\Z$ and $\alpha $, $\beta \in \hb (X,\Q)$ we have
\[\mathfrak{q}_{i}\he(\alpha )\mathfrak{q}_{j}\he(\beta )-(-1)^{|\alpha |\, |\beta |}\mathfrak{q}_{j}\he(\beta )\mathfrak{q}_{i}\he(\alpha )=i\, \delta _{i+j,0}\left( \int_{X}\alpha \beta\right)\id \].
\end{theorem}
\begin{proof}
 We adapt Nakajima's proof to our situation. The most interesting case is the computation of
$[\mathfrak{q}_{-i}\he(\alpha ),\mathfrak{q}_{j}\he(\beta )]$ when $i$ and $j$ are positive. We introduce the classes
$\bigl[ P_{ij}\he\bigr]$, resp. $\bigl[ Q_{ij}\he\bigr]$ in
\begin{align*}
 &\hs\bigl( \xna{n},\Q\bigr)\oti\hs\bigl( \xna{n-i},\Q\bigr)\oti\hs\bigl( \xna{n-i+j},\Q\bigr)\oti\hs\bigl( X,\Q\bigr)\oti\hs\bigl( X,\Q\bigr),\qquad\textrm{resp.}\\
&\hs\bigl( \xna{n},\Q\bigr)\oti\hs\bigl( \xna{n+j},\Q\bigr)\oti\hs\bigl( \xna{n-i+j},\Q\bigr)\oti\hs\bigl( X,\Q\bigr)\oti\hs\bigl( X,\Q\bigr),
\end{align*}
as follows:
\begin{align*}
\bigl[ P_{ij}\he\bigr]&:=p_{13*}\Bigl[ p_{124}\ee\bigl[ Q^{[n,n-i]\be}\bigr]\cdot p_{235}\ee \bigl[ Q^{[n-i+j,n-i]}\be\bigr]\Bigr],\quad\textrm{resp.}\\
\bigl[ Q_{ij}\he\bigr]&:=p_{13*}\Bigl[ p_{124}\ee\bigl[ Q^{[n+j,n]\be}\bigr]\cdot p_{235}\ee \bigl[ Q^{[n+j,n-i+j]}\be\bigr]\Bigr],
\end{align*}
where $Q^{[r,s]}\be$ is defined in (\ref{EqTrois}).
Then $\mathfrak{q}_{j}\he(\beta )\mathfrak{q}_{-i}\he(\alpha )$, resp. $\mathfrak{q}_{-i}\he(\alpha )\mathfrak{q}_{j}\he(\beta )$, is given by
\[
 \xymatrix@R=3pt{\tau \ar@{|->}[r]&{PD^{-1}\be\Bigl[ \pr_{3*}\he\bigl(\bigl[ P_{ij}\he\bigr]\cap \bigl( \pr_{5}\ee\beta \cup \pr_{4}\ee\alpha \cup \pr_{1}\ee\tau \bigr) \bigr)\Bigr]},& \hspace*{-6mm}\textrm{resp.}\\
 \tau \ar@{|->}[r]&{PD^{-1}\be\Bigl[ \pr_{3*}\he\bigl(\bigl[ Q_{ij}\he\bigr]\cap \bigl( \pr_{5}\ee\alpha \cup \pr_{4}\ee\beta \cup \pr_{1}\ee\tau \bigr) \bigr)\Bigr].}}
\]
First we deform all the relative integrable complex structures into a single one parametrized by $\snx\tim S^{n-i}\be X\tim S^{n-i+j}\be X\tim S^{2}\be X$.
Let $J_{n\tim(n-i)\tim(n-i+j)\tim 2}^{\, \re}$ be a relative integrable structure in a \nbh\ of
$Z_{n\tim(n-i)\tim(n-i+j)\tim 2}\he$, and
let
\[
Y=\Bigl( \xnb{[n]\tim[n-i]\tim[n-i+j]\tim [1]\tim [1]},J^{\re }_{n\tim(n-i)\tim(n-i+j)\tim 2}\leftarrow\textrm{5 times}\Bigr)
\]
be the product Hilbert scheme obtained by taking the same structure $J^{\re }_{n\tim(n-i)\tim(n-i+j)\tim 2}$ five times (see Definition \ref{DefUnSectionUnSubsectionTrois}),
where $J^{\re }_{n\tim(n-i)\tim(n-i+j)\tim 2}$ is identified with its pullback by
\[
\apl{\mu }{\snx\tim S^{n-i}\be X\tim S^{n-i+j}X\be \tim X\tim X}{\snx\tim S^{n-i}\be X\tim S^{n-i+j}X\be \tim S^{2}\be X.}
\]
Then, via the canonical isomorphism
\[
\hs\bigl( Y,\Q\bigr)\simeq\hs\bigl( \xn,\Q\bigr)\oti\hs\bigl( \xna{n-i},\Q\bigr)\oti\hs\bigl( \xna{n-i+j},\Q\bigr)\oti H_{*}\he \bigl( X,\Q\bigr)\oti H_{*}\he \bigl( X,\Q\bigr),
\]
since incidence varieties vary trivially in families,
the class $p_{124}\ee\bigl[ Q^{[n,n-i]}\be \bigr]$ is the homology class of the cycle
\[A=\Bigl\{(\xi ,\xi ',\xi'',s,t)\ \textrm{such that}\ \xi '\suq\xi \ \textrm{and}\ \rho (\xi ',\xi )=s\Bigr\}\cdot \]
In the same way, $p_{235}\ee\bigl[ Q^{[n-i+j,n-i]}\be \bigr]$ is the homology class of the cycle
\[
B=\Bigl\{(\xi ,\xi ',\xi'',s,t)\ \textrm{such that}\ \xi '\suq\xi ''\ \textrm{and}\ \rho (\xi ',\xi  '')=t\Bigr\}\cdot
\]
We now study the intersection of the cycles $A$ and $B$.
Let $p\in A\cap B$. We choose relative holomorphic coordinates
$\phi _{\xb,\yb,\zb,s,t}\he$ for $J^{\re }_{n\tim(n-i)\tim(n-i+j)\tim 2}$ such that
\[
\xymatrix{(\xb,\yb,\zb,s,t)
\ar@{|->}[r]&\bigl( S^{n}\be \phi _{\xb,\yb,\zb,s,t}\he(\xb),S^{n-i}\be \phi _{\xb,\yb,\zb,s,t}\he(\yb),S^{n-i+j}\be \phi _{\xb,\yb,\zb,s,t}\he(\zb),\phi _{\xb,\yb,\zb,s,t}\he(s),\phi _{\xb,\yb,\zb,s,t}\he(t)\bigr)
}
\]
is a local homeomorphism. The associated map given by
\[
\xymatrix{(\xi ,\xb,\xi ',\yb,\xi '',\zb,s,t)
\ar@{|->}[r]&\bigl(\phi _{\xb,\yb,\zb,s,t*}\he\xi , \phi _{\xb,\yb,\zb,s,t*}\he\xi ',S^{n-i+j}\be \phi _{\xb,\yb,\zb,s,t*}\he\xi '',\phi _{\xb,\yb,\zb,s,t}\he(s),\phi _{\xb,\yb,\zb,s,t}\he(t)\bigr)
}
\]
is a homeomorphism from a \nbh\ of $p$ to its image in $(\C^{2}\be)^{[n]}\be\tim(\C^{2}\be)^{[n-i]}\be\tim\he(\C^{2}\be)^{[n-i+j]}\be\tim \C^{2}\be\tim \C^{2}\be$ which maps $A$ and $B$ to the classical cycles $p^{-1}_{124}Q^{[n,n-i]}\be$ and $p^{-1}_{235}Q^{[n-i+j,n-i]}\be$. In the integrable case, we know that in the open set $\{s\not = t\}$, $p^{-1}_{124}Q^{[n,n-i]}\be$ and $p^{-1}_{235}Q^{[n-i+j,n-i]}\be$ intersect generically  transversally. Using relative holomorphic coordinates as above, this property still holds in our context. If $\bigl( A\cap B\bigr)_{s\not = t}\he=C_{ij}\he$, we can write $[A]\cdot[B]=
\bigl[\,  \ba{C{\vphantom{{'}}}_{ij}}\, \bigr]+\iota _{*}\he R$ where $\xymatrix{\iota :Y_{\{s= t\}}\he\ar@{^{(}->}[r]&Y
}$ is the natural injection and $R\in H_{2(2n-i+j+2)}\he\bigl( Y_{\{s= t\}}\he,\Q\bigr)$. We can do the same in
\[
Y'=\Bigl( \xnb{[n]\tim[n+j]\tim[n-i+j]\tim [1]\tim [1]},J^{\re }_{n\tim(n+j)\tim(n-i+j)\tim 2}\leftarrow\textrm{5 times}\Bigr)
\]
with the cycles $A'$ and $B'$ defined by
\begin{align*}
 A'&=\Bigl\{(\xi ,\xi ',\xi'',s,t)\ \textrm{such that}\ \xi \suq\xi ',\ \rho (\xi ,\xi' )=s\Bigr\}\\B'&=\Bigl\{(\xi ,\xi ',\xi'',s,t)\ \textrm{such that}\ \xi'' \suq\xi ',\ \rho (\xi '' ,\xi' )=t\Bigr\}\cdot
\end{align*}
We put $D_{ij}\he=\bigl( A'\cap B'\bigr)_{s\not = t}\he$. Then
$[A']\cdot[B']=
\bigl[\,  \ba{D{\vphantom{{'}}}_{ij}}\, \bigr]+\iota '_{*}R'$, where $\xymatrix{\iota ':Y'_{\{s= t\}}\ar@{^{(}->}[r]&Y'
}$ is the injection and $R'\in H_{2(2n-i+j+2)}\he\bigl( Y'_{\{s= t\}},\Q\bigr)$.
The class $R$ (resp. $R'$) can be chosen supported in $A\cap B\cap Y_{\{s=t\}}\he$
(resp. in $A'\cap B'\cap Y'_{\{s=t\}}$).
\par\medskip
The following lemma describes the situation outside the diagonal $\{s=t\}$.
\begin{lemma}\label{LemUn}
 $p_{1345*}\he\Bigl( \bigl[\,  \ba{C{\vphantom{{'}}}_{ij}}\, \bigr]\cap\bigl(\pr_{5}\ee\beta \cup \pr_{4}\ee \alpha \bigr)\Bigr)=(-1)^{|\alpha |\, |\beta |}\be p_{1345*}\he
\Bigl( \bigl[\,  \ba{D{\vphantom{{'}}}_{ij}}\, \bigr]\cap\bigl(\pr_{5}\ee\alpha \cup \pr_{4}\ee \beta \bigr)\Bigr)$.
\end{lemma}

\begin{proof}
 Let us introduce the incidence varieties
\begin{align*}
 T&=\Bigl\{(\xb,\yb,\zb,s,t)\in\snx\tim S^{n-i}\be X\tim S^{n-i+j}\be X\tim X\tim X\ \textrm{such that}\ \xb=\yb+ is,\ \zb=\yb+ jt\Bigr\}\\
T'&=\Bigl\{(\xb,\yb,\zb,s,t)\in\snx\tim S^{n+j}\be X\tim S^{n-i+j}\be X\tim X\tim X\ \textrm{such that}\ \yb=\xb+js=\zb+ it\Bigr\}
\end{align*}
Let $\Omega$, $\Omega '$ be two small \nbh s of $T$ and $T'$ and $W$ a \nbh\ of $Z_{n\tim(n-i+j)\tim 2}\he$ such that if $(\xb,\yb,\zb,s,t)\in \Omega $ (resp. $\Omega '$),
$\yb\in W_{\xb,\zb,s,t}\he$.
Let $J^{\re }_{n\tim(n-i+j)\tim 2}$ be a relative integrable complex structure on $W$.
After shrinking $\Omega $ and $\Omega '$ if necessary, we can consider
two relative structures
$J^{\re }_{n\tim(n-i)\tim(n-i+j)\tim 2}$ and $J^{\re }_{n\tim(n+j)\tim(n-i+j)\tim 2}$ such that
\[
\begin{cases}
 \forall (\xb,\yb,\zb,s,t)\in\Omega,\quad
J^{\re }_{n\tim(n-i)\tim(n-i+j)\tim 2,\xb,\yb,\zb,s,t}=J^{\re }_{n\tim(n-i+j)\tim 2,\xb,\zb,s,t}\\
\forall (\xb,\yb,\zb,s,t)\in\Omega',\quad
J^{\re }_{n\tim(n+j)\tim(n-i+j)\tim 2,\xb,\yb,\zb,s,t}=J^{\re }_{n\tim(n-i+j)\tim 2,\xb,\zb,s,t}
\end{cases}
\]
Let $U$ (resp. $U'$) be the points of $Y$ (resp. $Y'$) lying over $\Omega $ (resp. $\Omega '$). We define two maps $u$ and $v$ as follows:
\[
\xymatrix@R=3pt{
u:U\ar[r]&\bigl( \xnb{[n]\tim[n-i+j]\tim [1]\tim [1]},J^{\re }_{n\tim(n-i+j)\tim 2}\leftarrow \textrm{4 times}\bigr),\quad ( \xi ,\xi ',\xi '',s,t)\ar@{|->}[r]&(\xi ,\xi '',s,t),\\
v:U'\ar[r]&\bigl( \xnb{[n]\tim[n-i+j]\tim [1]\tim [1]},J^{\re }_{n\tim(n-i+j)\tim 2}\leftarrow \textrm{4 times}\bigr),\quad ( \xi ,\xi ',\xi '',s,t)\ar@{|->}[r]&(\xi ,\xi '',s,t)
}
\]
If we take
homeomorphisms between $\xn\!\tim\xna{n-i}\!\tim\xna{n-i+j}\!\tim X^{2}$, $\xn\!\tim\xna{n+j}\!\tim\xna{n-i+j}\!\tim X^{2}$, $\xn\!\tim\xna{n-i+j}\!\tim X^{2}$ and
$\xnb{[n]\tim[n-i]\tim[n-i+j]\tim[1]\tim[1]}$, $\xnb{[n]\tim[n+j]\tim[n-i+j]\tim[1]\tim[1]}$, $\xnb{[n]\tim[n-i+j]\tim[1]\tim[1]}$, $u$ and $v$ can be extended to global maps which are in the homotopy class of $p_{1345}\he$.
As in the integrable case, there is an isomorphism
$\xymatrix{\phi :C_{ij}\he\ar[r]^{\simeq}&D_{ij}\he}$ given as follows: if
$(\xi ,\xi ',\xi '',s,t)\in C_{ij}\he$ with $HC(\xi ')=\yb$, $HC(\xi )=\yb+is$, $HC(\xi '')=\yb+ jt$, then
$\phi (\xi ,\xi',\xi '',s,t )=(\xi ,\ti{\xi },\xi '',t,s)$ where $\ti{\xi }$ is defined by
$\ti{\xi }_{\, \vert\, p}\he={\xi '}_{ \vert\, p}\he$ if $p\in \yb$, $p\not \in\{s,t\} $, $\ti{\xi }_{\, \vert\, s}\he={\xi }_{\, \vert\, s}\he$
and $\ti{\xi }_{\, \vert\, t}\he={\xi ''}_{\, \vert\, t}\he$. All these schemes are considered for the structure $J^{\re }_{n\tim(n-i+j)\tim 2,\xb,\zb,s\, \cup \, t}$. Let $\pa C_{ij}\he=\ba{C_{ij}}\, \backslash C_{ij}\he$, $\pa D_{ij}\he=\ba{D_{ij}}\, \backslash D_{ij}\he$ and $S=u(\pa C_{ij}\he)=v(\pa D_{ij}\he)$.
We define $\apl{\pi }{Y'}{Y'}$ by $\pi (\xi ,\xi',\xi '',s,t )=(\xi ,\xi',\xi '',t,s )$.
We have the following diagram, where all the maps are proper:
\[
\xymatrix@C=3pt@R=30pt{Y\backslash\pa C_{ij}\he\supseteq C_{ij}\he\ar@{-}[rr]^{\phi }_{\simeq}\ar[dr]_{u}&&D_{ij}\he\suq Y'\backslash \pa D_{ij}\he\ar[dl]^{v\circ\pi} \\
&\xnb{[n]\tim[n-i+j]\tim [1]\tim [1]}\backslash S}.
\]
Thus we obtain in the Borel-Moore homology group $H^{\textrm{lf}}_{2(2n-i+j+2)}\bigl(\xnb{[n]\tim[n-i+j]\tim [1]\tim [1]} \backslash S,\Q\bigr)$ the equality
\[
u_{*}\he\bigl( [C_{ij}\he]\cap (\pr_{5}\ee\beta \cup \pr_{4}\ee\alpha )\bigr)=v_{*}\he\bigl( [D_{ij}\he]\cap (\pr_{4}\ee\beta \cup \pr_{5}\ee\alpha )\bigr).
\]
Since $\dim S\le 2(2n-i+j+2)-2$,
we get
 \[p_{1345*}\he\Bigl( \bigl[\,  \ba{C{\vphantom{{'}}}_{ij}}\, \bigr]\cap\bigl(\pr_{5}\ee\beta \cup \pr_{4}\ee \alpha \bigr)\Bigr)=(-1)^{|\alpha |\, |\beta |}\be p_{1345*}\he
\Bigl( \bigl[\,  \ba{D{\vphantom{{'}}}_{ij}}\, \bigr]\cap\bigl(\pr_{5}\ee\alpha \cup \pr_{4}\ee \beta \bigr)\Bigr).\]
\end{proof}
By this lemma, in $\bigl[\mathfrak{q}_{-i}\he(\alpha ),\mathfrak{q}_{j}\he(\beta )\bigr]$, the terms coming from
$\ba{C{\vphantom{{'}}}_{ij}}$ and $\ba{D{\vphantom{{'}}}_{ij}}$ cancel out. It remains to deal with the excess intersection components along the diagonals $Y_{\{s=t\}}\he$ and $Y'_{\{s=t\}}$. We introduce the locus
\begin{align*}
 \Gamma =\Bigl\{\bigl( \xi ,\, \xb,\, &\xi '',\zb,\, s,\, t\bigr)\in \xnb{[n]\tim[n-i+j]\tim [1]\tim [1]}\ \textrm{such that}\ s=t,\ \xi _{\, \vert\, p}\he=\xi ''_{\, \vert\, p}\ \textrm{for}\ p\not =s\\ &\textrm{and}\ HC(\xi '')=HC(\xi )+(j-i)s\ \textrm{if}\ j\ge i,\
HC(\xi )=HC(\xi '')+(i-j)s\ \textrm{if}\ j\le i\Bigr\}\cdot
\end{align*}
$\Gamma $ contains $u(A\cap B)$ and $v(A'\cap B')$. As before, the dimension count can be done as in the integrable case: $\dim\Gamma <2(2n-i+j+2)$ if $i\not = j$ and if $i=j$, $\Gamma $ contains a $2(2n+2)$-dimensional component, namely $\Delta _{\xn}\times\Delta _{X}\he$. All other components have lower dimensions. Thus, if $i\not = j  $, $p_{1345*}\he(\iota _{*}\he R)=0$ and $p_{1345*}\he(\iota '_{*} R')=0$ since they are supported in $\Gamma $ and have degree $2(2n-i+j+2)$. If $i=j$, then $p_{1345*}\he(\iota _{*}\he R)$ and $p_{1345*\he}(\iota '_{*} R')$ are proportional to the fundamental class of $\Delta _{\xn}\he\times\Delta _{X}\he$.
Now $p_{45*}\he\bigl( \bigl[ \Delta _{\xn}\tim\Delta _{X}\he\bigr]\cap \bigl( \pr_{4}\ee\alpha \cup \pr_{5}\ee\beta \bigr)\bigr)=\ds\int_{X}\alpha \beta .\bigl[ \Delta _{\xn}\bigr]$ and we obtain $\bigl[\mathfrak{q}_{-i}\he(\alpha ),\mathfrak{q}_{i}\he(\beta )\bigr]=\mu \ds\int_{X}\alpha \beta .\id$ where $\mu  \in\Q$.
The computation of the multiplicity $\mu $ is a local problem on $X$ which is solved in
\cite{SchHilGr}, \cite{SchHilStr}. It turns out that $\mu =-i$.
\end{proof}
\begin{remark}
The proof remains quite similar for $i>0$, $j>0$.
There is no excess term in this
\par\smallskip \noindent
case. Indeed, $Y=\xnb{[n]\tim[n+i]\tim[n+i+j]\tim [1]\tim [1]}$, $\Gamma =\xnb{[n+i+j,n]}\suq \xnb{[n]\tim[n+i]\tim[n+i+j]\tim [1]\tim [1]}$ and $\dim \Gamma =2(2n+i+j+1)<2(2n+i+j+2)$.
\end{remark}
Theorem \ref{PropUnSectionDeuxSubsectionDeux} gives a representation in $\HHH:=\bigoplus\limits_{n\in \N}\he H^{*}\be(\xn,\Q)$ of the Heisenberg super-algebra $\mathcal{H}(H\ee\be(X,\Q))$.
\begin{proposition}\label{PropUnSectionZorro}
 $\HHH$ is an irreducible $\hh(\hb (X,\Q))$-module generated by the vector 1.
\end{proposition}
This a consequence of Theorem \ref{PropUnSectionDeuxSubsectionDeux} and \got's formula (Theorem \ref{ThDeuxSectionUnSubsectionDeux}), as shown by Nakajima
\cite{SchHilNa}.


\section{Tautological bundles}\label{SecTrois}
\subsection{Construction of the tautological bundles}\label{SecTroisUn} Our aim in this section is to associate to any complex vector bundle $E$ on an almost-complex compact fourfold $X$ a collection of complex vector bundles
$ E^{[n]}\be$ on $\xn$ which generalize the tautological bundles already known in the algebraic context. The vector bundles $ E^{[n]}\be$ are constructed
using an auxiliary relative holomorphic structure on $E$. However, the classes $ E^{[n]}\be$ in $K(\xn)$ are canonical. Finally, we compare the classes $E^{[n]}\be$ and $ E^{[n+1]}\be$ in $K(\xn)$ and $K(X^{[n+1]}\be)$ through the incidence variety $X^{[n+1,n]}\be$.
\par\medskip
In the classical case, let $E$ be a algebraic vector bundle on an algebraic surface $X$. For $n\in\N$,
let $\apl{p}{\xn\tim X}{\xn}$ and $\apl{q}{\xn\tim X}{X}$ be the two projections and
let $Y_{n}\he\suq\xn\tim X$ be the incidence locus. Then
$\apl{p_{\, \vert\, Y_{n}\he}}{Y_{n}\he}{\xn}$ is finite. The tautological vector bundle $E^{[n]}\be$ is defined by $E^{[n]}\be=p_{*}\he\bigl( q\ee\be E\oti\oo_{Y_{n}\he}\bigr)$ and satisfies: for all $\xi $ in $\xn$,
$E_{\, \vert\, \xi }^{[n]}=H^{0}\be\bigl(\xi ,i_{\xi }\ee E\bigr)$. Our first aim is to generalize this construction in the almost-complex case.
\par\medskip
Let $(X,J)$ be an almost-complex compact fourfold, $Z_{n}\he
\suq \snx\times X$ the incidence locus, $W$ a small \nbh\ of $Z_{n}\he$ and $J_{n}^{\,  \rel}$ a relative integrable structure on $W$. The fibers of $\apl{\pr_{1}\he}{W}{\snx}$ are smooth analytic sets. We endow $W$ with the sheaf $\aaa_{W}\he$ of continuous functions which are smooth on the fibers of $\pr_{1}\he$. We can consider the sheaf $\aaa_{W,\rel}^{0,1}$ of relative $(0,1)$-forms on $W$.
There exists a relative
\mbox{$\ba{\pa\he}$-operator} $\apl{\ba{\pa\he}^{\, \rel}\be}{\aaa_{W}\he}{\aaa_{W,\rel}^{\, 0,1}}$ which induces for each $\xb\in\snx$ the usual operator $\apl{\ba{\pa\he}}{\aaa_{W_{\xb}\he}\he}{\aaa_{W_{\xb}\he}^{\, 0,1}}$ given by the complex structure $J_{n,\xb}^{\, \rel}$ on $W_{\xb}\he$.
\begin{definition}
Let $E$ be a complex vector bundle on $X$.
\begin{enumerate}
 \item [(i)] A \emph{relative connection $\ba{\pa\he}_{E}^{\, \rel}$ on $E$ compatible with $J_{n}^{\, \rel}$} is a $\C$-linear morphism of sheaves $\apl{\ba{\pa\he}_{E}\he}{\aaa_{W}\he\bigl( \pr_{2}\ee E\bigr)}{\aaa_{W}^{\, 0,1} \bigl( \pr_{2}\ee E\bigr)}$ satisfying
$\ba{\pa\he}_{E}^{\, \rel}(\varphi s)=\varphi\,  \ba{\pa\he}_{E}^{\, \rel} s+\ba{\pa\he}\be^{\, \rel}\varphi \oti s$ for all sections $\varphi $ of $\aaa_{W}\he$ and $s$ of  $\aaa_{W}\he\bigl( \pr_{2}\ee E\bigr)$.
\item[(ii)] A relative connection $\ba{\pa\he}_{E}^{\, \rel}$ is \emph{integrable} if
$\bigl( \ba{\pa\he}_{E}^{\, \rel}\bigr)^{2}\be=0$.
\end{enumerate}
\end{definition}
If $\ba{\pa\he}_{E}^{\, \rel}$ is an integrable  connection on $E$ compatible with
$J_{n}^{\, \rel}$, we can apply the Kozsul-Malgrange integrability theorem with continuous parameters in $\snx$ (see \cite{SchHilVo3}). Thus, for every $\xb\in\snx$, $E_{\vert W_{\xb}\he}$ is endowed with the structure of a holomorphic vector bundle over $\bigl( W_{\xb}\he,J_{n,\xb}^{\, \rel}\bigr)$ and this structure varies continuously with $\xb$. Furthermore, $\ker\ba{\pa\he}_{E}^{\, \rel}$ is the sheaf of relative holomorphic sections of $E$. Therefore, there is no difference between relative integrable connections on $E$ compatible with $J_{n}^{\, \rel}$ and relative holomorphic structures on $E$ compatible with $J_{n}^{\, \rel}$.
\par\medskip
Taking relative holomorphic coordinates for $J_{n}^{\, \rel}$, we can see that relative integrable connections exist on $W$ over small open sets of $\snx$. By a partition of unity on $\snx$, it is possible to build global ones. The space of holomorphic structures on a complex vector bundle over a ball in
$\C^{2}\be$ is contractible. Therefore the space of relative holomorphic structures on $E$ compatible with $J_{n}^{\, \rel}$ is also contractible.
\par\medskip
We proceed now to the construction of the tautological bundle $E^{[n]}\be$ on $\xnj{n}$. Let $\ba{\pa\he}_{E}^{\, \rel}$ be a relative holomorphic structure on $E$ adapted to $J_{n}^{\, \rel}$. Taking relative holomorphic coordinates, we get a vector bundle $E^{[n]}_{\, \rel}$ over $W^{[n]}_{\rel}$ satisfying: for each $\xb$ in $\snx$, $E^{[n]}_{\, \rel\vert W_{\xb}^{[n]}}=E^{[n]}_{\vert W_{\xb}\he}$, where $E_{\vert W_{\xb}\he}\he$ is endowed with the holomorphic structure given by $\ba{\pa\he}_{E,\xb}^{\, \rel}$.
\begin{definition}\label{DefDeux}
Let $\apl{i}{\xnj{n}}{W_{\re }^{\, [n]}}$ be the canonical injection. The complex vector bundle $\bigl( \enn,J^{\re }_{n},\ba{\pa\he}_{E}^{\, \rel}\bigr)$ on $\xnj{n}$ is defined by $E^{\, [n]}\be=i\ee\be E^{\, [n]}_{\re }$.
\end{definition}
In the sequel, we consider the class of $\enn$ in $K\bigl( \xn\bigr)$, which we prove below to be independent of the structures used in the construction.
\begin{proposition}\label{prop}
 The class of $\enn$ in $K\bigl( \xn\bigr)$ is independent of  $\bigl( J^{\re }_{n},\ba{\pa\he}_{E}^{\, \rel}\bigr)$.
\end{proposition}
\begin{proof}
 Let $\bigl( J_{0,n}^{\re },\ba{\pa\he}_{E, 0}^{\, \rel}\bigr)$ and $\bigl( J_{1,n}^{\re },\ba{\pa\he}_{E, 1}^{\, \rel}\bigr)$ be two relative holomorphic structures on $E$,
$\bigl( J_{t,n}^{\re },\ba{\pa\he}_{E,t}^{\, \rel}\bigr)$ be a smooth path between them, and
$W_{\re }^{\, [n]}$ be the relative Hilbert scheme over $\snx\tim [0,1]$ for the family $\bigl( J_{t,n}^{\re }\bigr)_{0\le t\le 1}\he$. There exists a vector bundle
$\bigl( \ti{E}_{\re }^{\, [n]},\bigl\{J^{\re }_{t,n}\bigr\}_{0\le t\le 1}\he,
\bigl\{\ba{\pa\he}_{E, t}^{\, \rel}\bigr\}_{0\le t\le 1}\he\bigr)$ over $W_{\re }^{\, [n]}$ such that for all $t$ in
$[0,1]$, $\ti{E}^{\, [n]}_{\rel\vert\, W_{\rel,t}^{\, [n]}}=\bigl( {E}_{\re }^{\, [n]},J^{\re }_{t,n},
\ba{\pa\he}_{E, t}^{\, \rel}\bigr)$. If $\xg  =\bigl( \xn,\bigl\{J_{t,n}^{\re }\bigr\}_{0\le t\le 1}\he\bigr)\suq W_{\re }^{\, [n]}$ is the relative Hilbert scheme over $[0,1]$, then $\ti{E}^{\, [n]}_{\re \vert\, \xg  }$ is a complex vector bundle on $\xg  $ whose restriction to $\xg  _{t}\he$ is $\bigl( E^{\, [n]}\be, J^{\re }_{t,n}, \ba{\pa\he}_{E, t}^{\, \rel}\bigr)$. Now $\xg  $ is topologically trivial over $[0,1]$ by Proposition \ref{PropDeuxSectionUnsubsectionDeux}. Since $K(\xg  _{0}\he\tim[0,1])\simeq K(\xg  _{0}\he)$, we get the result.
\end{proof}
If $\T=X\tim\C$ is the trivial complex line bundle on $X$, the tautological bundles
$\T^{[n]}\be$ already convey geometric informations on $\xn$. Let $\pa \xn\suq\xn$ be the inverse image of the big diagonal of $\snx$ by the Hilbert-Chow morphism. We have $\dim\pa \xn=4n-2$ and
$H_{4n-2}\he(\pa\xn,\Z)\simeq\Z$ (this can be proved as in Lemma \ref{LemUnSectionDeuxSubsectionUn}).
\begin{lemma}\label{lem}
 $c_{1}\he(\T^{[n]}\be)=-\frac{1}{2}\, PD^{-1}\be\bigl( \bigl[ \pa\xn\bigr]\bigr)$ in
$H^{2}\be\bigl( \xn,\Q\bigr)$.
\end{lemma}
\begin{proof}
Let $U=\bigl\{(x_{1},\dots,x_{n}\he)\in\xn\ \textrm{such that for all}\ (i,j)\ \textrm{with}\ i\neq j, \ x_{i}\neq x_{j}\bigr\}$. Then $\xn\setminus \partial\xn$ is canonically isomorphic to $ U/\mathfrak{S}_{n}\he$. If $\apl{\sigma }{U}{\xn\setminus \partial\xn}$ is the associated quotient map, $\sigma \ee\be\T^{[n]}\be\simeq\bop_{i=1}^{n}\pr_{i}\ee\T$, so that $\sigma \ee\be\T^{[n]\be}$ is trivial. Since $\sigma $ is a finite covering map, $c_{1}(\T^{[n]})_{\vert\xn\setminus\partial\xn}\he$ is a torsion class, so it is zero in
$H^{2}\bigl( \xn\setminus\partial\xn,\Q\bigr)$.
This implies that $c_{1}\he(\T^{[n]}\be)=\mu \, PD^{-1}\be\bigl( [\pa\xn]\bigr)$ where $\mu \in\Q$.
To compute $\mu $, we argue locally on $\snx$ around a point in the stratum
\[
S=\bigl\{\xb\in\snx\ \textrm{such that}\ x_{i}\he\not =x_{j}\he\ \textrm{except for one pair} \ \{i,j\}\bigr\}\cdot
\] This reduces the computation to the case $n=2$. Then $U^{[2]}\be =Bl_{\Delta }\he(U\tim U)/\Z_{2}\he$, where $U\suq X$ is endowed with an integrable complex structure and $\Delta $ is the diagonal of $U$. If $E\suq Bl_{\Delta }\he(U\tim U)$ is the exceptional divisor and
$\apl{\pi }{Bl_{\Delta }\he(U\tim U)}{U^{[2]}\be}$ is the projection, then
$\pi \ee\be \bigl( [\pa U^{[2]}]\bigr)=2[E]$ and $\pi \ee\be c_{1}\he(\T^{[2]})=c_{1}(\pi \ee\be \T^{[2]}\be)=c_{1}\he\bigl( \oo(-E)\bigr)=-[E]$ in $H^{2}\be\bigl(Bl_{\Delta }\he(U\tim U),\Z \bigr)$. This gives the value $\mu =-1/2$.
\end{proof}
\subsection{Tautological bundles and incidence varieties}
We want to compare the tautological bundles $E^{[n]}\be$ and $E^{[n+1]}\be$ through the incidence variety $\xna{n+1,n}$. In the integrable case, $\xna{n+1,n}$ is smooth. If $D\suq \xna{n+1,n}$ is the divisor $\ba{Z}_{1}\he$ (see (\ref{EqDeux})), we have an exact sequence (see \cite{SchHilDa}, \cite{SchHilLe}):
\begin{equation}\label{EqEtoile}
 \xymatrix{
0\ar[r]&\rho \ee\be E\oti\oo_{\xna{n+1,n}}\he(-D)\ar[r]&\nu  \ee\be E^{[n+1]}\be\ar[r]&\lambda \ee\be E^{[n]}\be\ar[r]&0,
}
\end{equation}
where $\apl{\lambda }{\xna{n+1,n}}{\xn}$, $\apl{\nu }{\xna{n+1,n}}{\xna{n+1}}$ and $\apl{\rho }{\xna{n+1,n}}{X}$ are the two natural projections and the residual map.
\par
In the almost-complex case, $\xna{n+1,n}$ is a topological manifold of dimension $4n+4$. If we choose a relative integrable structure $J^{\re }_{n+1}$ with additional properties as given in \cite{SchHilVo1}, $\xna{n+1,\, n}$ can be endowed with a differentiable structure, but we will not need it here.
\par
 Let $J^{\re }_{n}$ and $J^{\re }_{n+1}$ be two relative integrable structures in small \nbh s of $Z_{n}\he$ and $Z_{n+1}\he$. We extend them to relative structures $\check{J}_{n}^{\re}$ and $\check{J}_{n+1}^{\re}$ in small \nbh s of $Z_{n\tim(n+1)}\he$. Then
$\bigl( \xnb{[n]\tim[n+1]},\check{J}^{\re }_{n},\check{J}^{\re }_{n+1}\bigr)=
\xnj{n}\tim X^{[n+1]}_{J^{\re}_{n+1}}$. If $J^{\re }_{n\tim(n+1)}$ is a relative integrable structure in a small \nbh\ of $Z_{n\tim(n+1)}\he$
and $J^{\re }_{n\tim 1}$ is defined by $J^{\re }_{n\tim 1,\xb,p}=J^{\re }_{n\tim (n+1),\, \xb,\, \xb\, \cup\, p}$, then we have a diagram:
\[
\xymatrix@C=50pt@R=40pt{
&&X^{[n+1]}_{J^{\re}_{n+1}}\\
X^{[n+1,n]}_{J^{\re}_{n\tim 1}}\ar@/^15pt/[urr]^{\nu }\ar@{^{(}->}[r]\ar@/_15pt/[drr]_{\lambda }&\bigl( \xnb{[n]\tim[n+1]},J^{\re }_{n\tim(n+1)},J^{\re }_{n\tim(n+1)}\bigr)\ar[r]^(.55){\Phi }_(.55){\simeq}&\bigl( \xnb{[n]\tim[n+1]},\check{J}^{\re }_{n},\check{J}^{\re }_{n+1}\bigr)\ar[u]_{\pr_{1}\he}\ar[d]^{\pr_{2}\he}\\
&&\xnj{n}
}
\]
where $\Phi $ is a homeomorphism uniquely determined up to homotopy. We will denote by $D$ the inverse image of the incidence locus of $\snx\tim X$
by the map $\flgd{\xna{n+1,n}}{\snx\tim X,}$ so that  $D=\ba{Z}_{1}\he$ where $Z_{1}\he$ is defined by (\ref{EqDeux}). The cycle $D$ has a fundamental homology class in $H_{4n+2}\he\bigl( \xna{n+1,n},\Z\bigr)$. Furthermore, there exists a unique complex line bundle $F$ on $\xna{n+1,n}$ such that $PD\bigl( c_{1}\he(F)\bigr)=-[D]$.
\begin{proposition}\label{pprop}
 In $K\bigl( \xna{n+1,n}\bigr)$, the following identity holds: $\nu \ee\be E^{[n+1]}\be =\lambda \ee E^{[n]}\be +\rho \ee\be E\oti F.$
\end{proposition}
\begin{proof}
Let $\ba{\pa\he}^{\, \rel}_{E,n\tim 1}$, $\ba{\pa\he}^{\, \rel}_{E,n}$ and $\ba{\pa\he}^{\, \rel}_{E,n+1}$ be relative holomorphic structures on $E$ compatible with $J_{n\tim 1}^{\, \rel}$,
$J_{n}^{\, \rel}$ and $J_{n+1}^{\, \rel}$.
 For each $(\xb,p)\in\snx\tim X$, we consider the exact sequence (\ref{EqEtoile}) on $\bigl( W_{\xb,p}\he,J^{\, \rel}_{n\tim 1,\xb,p}\bigr)$ for the holomorphic vector bundle
$\bigl( E_{\vert W_{\xb,p}\he}\he,\ba{\pa\he}^{\, \rel}_{E,n\tim 1,\xb,p}\bigr)$.
Putting these exact sequences in families over $\snx\tim X,$ and restricting it to $\xna{n+1,n}$, we get an exact sequence
$\sutrgdpt{\rho \ee\be E\oti G}{A}{B}{,}$ where $G$ is a complex line bundle on $\xna{n+1,n}$ and $A$ and $B$ are two vector bundles on $\xna{n+1,n}$ such that for all $ (\xb,p)$ in $\snx\tim X$:
\begin{equation}\label{EqCinq}
A_{\vert\, \xi ,\xi ',\xb,p}\he=\Bigl( E_{\vert\, \xi '}^{[n+1]},
\ba{\pa\he}^{\, \rel}_{E,n\tim 1,\xb,p},J^{\re }_{n\tim 1,\xb,p}\Bigr),\quad
B_{\vert\, \xi ,\xi ',\xb,p}\he=\Bigl( E_{\vert\, \xi }^{[n]},
\ba{\pa\he}^{\, \rel}_{E,n\tim 1,\xb,p},J^{\re }_{n\tim 1,\xb,p}\Bigr).
\end{equation}
Now, $\Phi $ is given by $\Phi (\xi ,\xi ',\ub{u},\ub{v})=\bigl( \phi _{\ub{u},\ub{v}*}\he\xi ,S^{n}\be\phi _{\ub{u},\ub{v}}\he(\ub{u}),\psi _{\ub{u},\ub{v}*}\he\xi ',S^{n+1}\be \psi _{\ub{u},\ub{v}*}\he(\ub{v})\bigr)$. Thus
\begin{align*}
\nu \ee\be E^{[n+1]}_{\vert\, \xi ,\xi ',\xb,p}&=\Bigl( E^{[n+1]}_{\vert\, \psi _{\xb,\xb\, \cup\, p*}\he \xi '},\ba{\pa\he}^{\, \rel}_{E,n+1,S^{n+1}\be\psi _{\xb,\xb\, \cup\, p}\he(\xb\, \cup\, p)},{J}^{\re }_{n+1,S^{n+1}\be \psi _{\xb,\xb\, \cup\, p}\he(\xb\, \cup\, p)}\Bigr),\\
\lambda  \ee\be E^{[n]}_{\vert\, \xi ,\xi ',\xb,x}&=\Bigl( E^{[n]}_{\vert\, \phi _{\xb,\xb\, \cup\, p*}\he\xi },\ba{\pa\he}^{\, \rel}_{E,n,S^{n}\be\phi _{\xb,\xb\, \cup\, p}\he(\xb)},{J}^{\re }_{n,S^{n}\be \phi _{\xb,\xb\, \cup\, p}\he(\xb)}\Bigr).
\end{align*}
As in Proposition \ref{prop}, the classes $A$ and $B$ in $K\bigl( \xna{n+1,n}\bigr)$ are independent of  the structures used to define them. If $J^{\re }_{n\tim(n+1)}=\check{J}^{\re }_{n+1}$ and for all $(\xb,p)$ in $\snx\tim X$,
$\ba{\pa\he}^{\, \rel}_{E,n\tim 1,\xb,p}=\ba{\pa\he}^{\, \rel}_{E,n\tim 1,\xb\, \cup\,  p}$, we can take $\psi _{\ub{u},\ub{v}}\he=\id$ in a \nbh\ of $\ub{v}$. Thus $A=\nu \ee\be E^{[n+1]}\be$. On the other way, if $J^{\re }_{n\times(n+1)}=\check{J}^{\re }_{n}$ and
for all
$(\xb,p)$ in $\snx\tim X$,
$\ba{\pa\he}^{\, \rel}_{E,n\tim 1,\xb,p}=\ba{\pa\he}^{\, \rel}_{E,n,\xb}$ in a \nbh\ of $\xb$, we can take $\phi _{\ub{u},\ub{v}}\he=\id$ in a \nbh\ of $\ub{u}$. Thus
$B=\lambda \ee\be E^{[n]}\be$. This proves that $\nu \ee\be E^{[n+1]}\be-\lambda \ee\be E^{[n]}\be=\rho \ee\be E\oti G$ in $K\bigl( \xna{n+1,n}\bigr)$.
If $\T$ is the trivial complex line bundle on $X$, $\nu \ee\be \T^{[n+1]}\be\simeq\lambda \ee\be \T^{[n]}\be\oplus\rho \ee\be \T$ on $\xna{n+1,n}\backslash D$. Thus $G$ is trivial outside $D$. This yields
$PD(c_{1}\he(G))=\mu [D]$, where $\mu \in\Q$ and the computation of $\mu $ is local, as in Lemma \ref{lem}, so that $\mu =-1$.
\end{proof}
If $X$ is a projective surface, the subring of $H\ee\be(\xn,\Q)$ generated by the classes
$\ch_{k}\he(E^{[n]}\be)$ (where $E$ runs through all the algebraic vector bundles on $X$) is called the \emph{ring of algebraic classes of $\xn$}. If $(X,J)$ is an almost-complex compact fourfold, we can in the same manner consider the subring of $H\ee\be(\xn,\Q)$ generated by the classes $\ch_{k}\he(E^{[n]}\be)$, where $E$ runs through all the complex vector bundles on $X$. If $X$ is projective, this ring is much bigger than the ring of the algebraic classes. In a forthcoming paper, we will show that it is indeed equal to $H\ee\be(\xn,\Q)$ if $X $ is a symplectic compact fourfold satisfying $b_{1}(X)=0$, and we will describe the ring structure of $H\ee\be(\xn,\Q)$.


\end{document}